\newcommand{\vekk}[1]{}

\documentclass[preprint,12pt]{imsart}

\usepackage{amsthm,amsmath}
\usepackage{amsfonts}
\usepackage{amssymb}

\usepackage[a4paper,includeheadfoot,margin=2.54cm]{geometry}

\startlocaldefs

\usepackage{amsthm}
\usepackage{amssymb}
\usepackage{amsmath}

\newtheorem{theo}{Theorem}
\newtheorem{pro}{Proposition}

\newtheorem{lemma}{Lemma}
\newtheorem{defin}{Definition}

\vekk{
\newenvironment{proof}{{\it Proof.}}%
        {\ifvmode\else\unskip\fi ~\penalty10000 \hfill%
        $\Box$\vspace{0ex}}
}

\newcommand{\abs}[1]{ {\left | #1 \right |} }

\newcommand{\dlik}{\mathop{:=}\nolimits}

\newcommand{\Hi}{{\cal H}}
\newcommand{\id}{\mathop{id}\nolimits}
\newcommand{\im}{\mathop{Im}\nolimits}
\newcommand{\imply}{\;\; \Rightarrow \;\;}
\newcommand{\into}{\mbox{$\: \rightarrow \:$}}
\newcommand{\norm}[1]{\left \| #1\right\|}

\newcommand{\RealN}{{\mbox{$\mathbb R$}}} 
\newcommand{\st}{\mid}

\newcommand{\spec}{\mathop{Sp}\nolimits}

\endlocaldefs

\begin{document}

\begin{frontmatter}

\title{Nonlinear probability\\ 
     {\small A theory with incompatible stochastic variables}}

\runtitle{Nonlinear probability}

\begin{aug}
\author{\fnms{Gunnar} \snm{Taraldsen}\corref{}\ead[label=e1]{Gunnar.Taraldsen@ntnu.no}}

\address{Trondheim, Norway.
\printead{e1}}

\runauthor{Taraldsen}

\affiliation{Norwegian University of Science and Technology}

\end{aug}

\begin{abstract}
In 1991 J.F. Aarnes introduced the concept of quasi-measures
in a compact topological space $\Omega$ and
established the connection between quasi-states on
$C (\Omega)$ and quasi-measures in $\Omega$.
This work solved the linearity problem of quasi-states
on $C^*$-algebras formulated by R.V. Kadison in 1965.
The answer is that a quasi-state need not be linear,
so a quasi-state need not be a state.
We introduce nonlinear measures in a space
$\Omega$ which is a generalization of a 
measurable space.
In this more general setting we are still able to
define integration and establish a representation theorem
for the corresponding functionals.
A probabilistic language is choosen since we feel that the
subject should be of some interest to probabilists.
In particular we point out that the theory allows for
incompatible stochastic
variables.
The need for incompatible variables is well known in
quantum mechanics,
but the need  seems natural also in other contexts
as we try to explain by a questionary example. 
\end{abstract}

\begin{keyword}
\kwd{Epistemic probability}
\kwd{Integration with respect to measures and other set functions}
\kwd{Banach algebras of continuous functions}
\kwd{Set functions and measures on topological spaces}
\kwd{States}
\kwd{Logical foundations of quantum mechanics}
\end{keyword}

\tableofcontents

\end{frontmatter}

\newpage

\section{Incompatible experiments.\label{SCOMP}}

The purpose of this section is to indicate a general
model for probability theory \cite{TARALDSEN14} which is more general
than the model formulated by A. Kolmogorov in 1933 \cite{KOLMOGOROV}.  
A realization of this is given by the quantum mechanics 
formulated by J. von Neumann in 1932 \cite{NEU:QM},
but also by the theory of nonlinear probability formulated
in the following sections in this paper.
The main point is that the Kolmogorov model does not allow
for incompatible experiments. 
Incompatible experiments are central in the theory of
nonlinear probabilities.

Let $\Omega_X$ be the set of possible results $x$ of an experiment $X$.
The distribution $P_X$ gives the probability $P_X (A)$
for the event that the experiment gives an outcome $x \in A$.
Assume that $Y$ is an experiment which gives the result $y = \phi (x)$
whenever $X$ has the outcome $x$.
{\em In this case we say that the experiment $Y$ is
included in the experiment $X$},
and we write $Y = \phi (X)$.
The crucial observation now is that $y \in B$
is equivalent with $x \in \phi^{-1} (B)$,
and the probabilistic interpretation gives
the equality 
$P_Y (B) = P_X (\phi^{-1} (B))$.
In a more compact notation we have
\begin{equation}\label{EINCL}
 P_Y = P_X \circ \phi^{-1} .
\end{equation}
Let $\cal O$ be the family of experiments in a given model.
The family of experiments included in the experiment $X$
is denoted by ${\cal O}_X$.
{\em The experiments $X$ and $Y$ are said to be compatible
if there exists a third experiment $Z$ which contains 
both $X$ and $Y$},
or equivalently $X,Y \in {\cal O}_Z$.

In the Kolmogorov formulation $\cal O$ is given by the family of measurable functions 
$X: \Omega \into \RealN$  defined on a probability space $\Omega$.
It is then assumed that the result of an experiment can be represented by a real number.
This will also be assumed in the generalisation presented in the following.

The distribution $P_X$ of $X$ is given by
$P_X (A) = P (X \in A) \dlik P (X^{-1} (A))$, or 
\begin{equation}
P_X = P \circ X^{-1} .
\end{equation}
In this model all experiments become compatible:
Given two stochastic variables $X,Y$,
there exists a stochastic variable $Z$ and two measurable
functions $\phi_x, \phi_y : \RealN \into \RealN$ such that
$X = \phi_x (Z)$ and $Y = \phi_y (Z)$.

There are situations that demand models with incompatible experiments.
The most famous example is probably given in quantum mechanics.
The experiments $P$ and $Q$ corresponding to the measurement of 
a particles momentum and position are incompatible.
It is an empirical fact that it is impossible to construct an experiment
$Z$ which effects a simultaneous measurement of both $P$ and $Q$ \cite{JAUCH}.
In the von Neumann formulation of quantum mechanics $\cal O$ is
given by the family of selfadjoint operators $X$ on a Hilbert space $\Hi$.
The family ${\cal O}_X$ of experiments $\phi (X)$ contained in $X$
is defined by the functional calculus of $X$ \cite{NEU:QM,TARALDSEN14}.

Kolmogorov's formulation of probability theory is not general
enough for the formulation of quantum mechanics because every
experiment in the Kolmogorov model are compatible.
We will indicate another situation in which incompatible
experiments seem to appear naturally.

Consider a questionary.
Let two questions a) and b) be given.
Assume that the questions can only be answered with a yes or a no.
Let experiment $X$ be that question a) is posed to a person
and the outcome is registered as $x = 1$ if the person answers yes
and as $x = 0$ in the opposite case.
This gives $\Omega_X = \{0, 1\}$.
The distribution $P_X$
is interpreted as an a'priori judgement of the outcome.
Let experiment $Y$ be given in the same way by posing question b)
to an identical person.
The claim now is that there are no general reasons which imply
that experiments $X$ and $Y$ are compatible.
Consider the experiment $Z$ given by posing both question a) and question b)
to a third identical person.
At first sight it may seem that the experiment $Z$ contains both $X$ and $Y$.
The point is, however, that the answer to question a) need not be the same
in experiments $X$ and $Z$.
In experiment $Z$ the person also has to consider question b),
and this may influence the consideration of question a)
beyond the effects of correlations.

The main point in this section is that it seems to restrictive
to only consider probabilistic models with compatible experiments.
Quantum mechanics gives one possible generalization of
the Kolmogorov probability theory.
The theory of nonlinear measures is closer to Kolmogorov's probability
theory and gives models with incompatible experiments.

Aarnes and Rustad \cite{AarnesRustad99qprob} give a different interpretation and refer to an earlier version of 
the present article.

\section{Main results and definitions.\label{SDEF}}

In this paper we  establish a theory
of integration with respect to a nonlinear probability $P$ on
an arbitrary nonempty set $\Omega$.
The theory is a generalization of the theory of quasi-measures
on topological spaces $\Omega$ \cite{AARNES:QUASI}
and at the same time  a generalization of the theory of measures
on measurable spaces $\Omega$ \cite{RU:ANAL}. 
The main result is a representation theorem 
for corresponding functionals $E$.
This section concludes with the statement of
the representation theorem,
but first we need
some preliminary definitions.
We find it convenient to take the function class corresponding
to the family of measurable functions as our starting point.

\begin{defin}
A family $\cal O$ of functions $X: \Omega \into \RealN$
is an algebra of observables on $\Omega$ if
$X_1, \ldots , X_n \in {\cal O}$ and 
$\phi \in C (\im (X_1,\ldots , X_n ))$ implies
that $\phi (X_1, \ldots , X_n) \in {\cal O}$.
\end{defin}

In the above $C (\im (X_1,\ldots , X_n ))$ is the family of real
valued continuous
functions on the image 
$\im (X_1,\ldots , X_n ) = \{(X_1 (\omega),\ldots , X_n (\omega) ) \st 
\omega \in \Omega\}$
of the function $(X_1,\ldots , X_n )$.
The function  $\phi (X_1, \ldots , X_n)$ is defined by
$\phi (X_1, \ldots , X_n) (\omega) = \phi (X_1 (\omega), \ldots , X_n (\omega))$.
The case studied by J.F. Aarnes \cite{AARNES:QUASI} is obtained by letting
$\cal O$ be the family of real valued continuous functions on a
compact Hausdorff space $\Omega$.
We will refer to this as the compact case.
An algebra of observables on $\Omega$ is in particular an
algebra of real functions on $\Omega$ with the conventional
addition and multiplication of functions. 

A vector-observable $Y$ is
a tuple $Y = (X_1, \ldots, X_n)$ of observables.
The algebra of observables ${\cal O}_Y$ generated by a vector-observable
$Y$ consists of all functions $\phi (Y)$ where
$\phi \in C (\im (Y))$.
The defining property of an algebra of observables is then that
it contains every algebra of observables generated by the
associated vector-observables.
It should be observed that the above also can be used 
to define the algebra of observables generated by any
given family of real valued functions on an arbitrary set $\Omega$.
The family of bounded observables 
${\cal O}^{\infty} = \{X \in {\cal O} 
\st \norm{X} \dlik \sup_\omega \abs{X (\omega)} < \infty \}$
is not in general an algebra of observables,
but it is a normed algebra of functions.
The spectrum of a vector-observable $Y$ is the closure of its image,
$\spec Y = \overline{\im Y}$.
The closed algebra ${\cal A}_Y$ generated by a vector-observable
$Y$ consists of all functions $\phi (Y)$,
$\phi \in C (\spec (Y))$.
The inclusion ${\cal A}_Y \subset {\cal O}_Y$ may be strict.

\begin{defin}
An observable expectation is a functional $E: {\cal O}^\infty \into \RealN$
with the properties
\begin{description}
\item[{\bf (I)}] 
$E$ is quasi-linear: It is 
linear on each ${\cal A}_X$, $X \in {\cal O}^\infty$.
\item[{\bf (II)}] $E$ is positive:
   $X \ge 0 \imply E [X] \ge 0$.
\item[{\bf (III)}] $E$ is normalized: $E [1] = 1$.
\item[{\bf (IV)}] 
$E$ enjoys the monotone convergence property:
If $X_i \uparrow X$, then $E [X_i] \uparrow E [X]$. 
\end{description}
\end{defin}

The notation $X_i \uparrow X$ means that
$X_{i} \le X_{i + 1}$ and $X = \lim X_i$ pointwise.
A quasi-additive expectation is a functional 
$E: {\cal O}^\infty \into \RealN$ which obeys the axioms { (I-III)}.
In general the class of  quasi-additive expectations contains
the  class of observable expectations,
but in the compact case there is no difference and our definition
agrees with the original definition by J.F. Aarnes.
We will refer to axiom { (IV)} as the monotone convergence axiom,
because it is true for ordinary integrals by the Lebesgue monotone
convergence theorem \cite{RU:ANAL}. 
In the compact case the monotone convergence axiom holds even
for monotone nets of functions.

Let ${\cal O}$ be an algebra of observables on $\Omega$.
A subset $F$ of $\Omega$ is a zero set if 
$F = X^{-1} \{0\}$ for some observable $X$. 
The family of zero sets is denoted by ${\cal E}_F$ and the family
${\cal E}_U$ of co-zero sets is the family of the complements of the zero sets,
${\cal E}_U \dlik \{U \st U^c \in {\cal E}_F\}$.
The union of the two families is the family ${\cal E}$ of
observable sets.
In $\cal E$ we define the operation $\uplus$ such that $A \uplus B$
is the union of $A,B \in {\cal E}$, but it is only defined
when $A \cup B \in {\cal E}$ and $A \cap B = \emptyset$.
The axioms in the following definition of an observable probability
are identical with the axioms in the definition of a 
probability measure, but the $\sigma$-field is replaced by
the family of observable sets:

\begin{defin}
\label{DPROB}
An observable probability is a set function 
$P: {\cal E} \into \RealN$ such that
\begin{description}
\item[{\bf (I)}] 
$P$ is additive: $P (A \uplus B) = P (A) + P (B)$.
\item[{\bf (II)}] $P$ is positive: $P (A) \ge 0$.
\item[{\bf (III)}] $P$ is normalized: $P (\Omega) = 1$.
\item[{\bf (IV)}] 
$P$ enjoys the monotone convergence property:
If $U_i \uparrow U$ in ${\cal E}_U$, then $P (U_i) \uparrow P (U)$. 
\end{description}
\end{defin}

The notation $U_i \uparrow U$ means that
$U_i \subset U_{i + 1}$ and $U = \cup_i U_i$,
and it is assumed that $U_i, U \in {\cal E}_U$.
We keep the term nonlinear probability as a synonym to the 
term observable probability,
since the term quasi-measure is established in the literature.
J.F. Aarnes has suggested a third convention:
A probability $P$ on a family of observable sets $\cal E$ is 
defined to be an observable probability as above. 
A quasi-additive probability is a set function $P$
which obeys the axioms { (I-III)}, and  
\begin{description}
\item[{\bf (IV)'}] 
If $U$ is a co-zero set, then 
$P (U) = \sup \{P (F) \st F \subset U,\; F \in {\cal E}_F\}$.
\end{description}
Our results imply that the class of 
quasi-additive probabilities contains the class of nonlinear probabilities,
but in the compact case the classes are equal.
In the next section we define integration
with respect to an observable probability,
and show that the integral gives an observable expectation.
The converse is our main result:
\begin{theo}\label{TREP}
If $E$ is an observable expectation on $\cal O$,
there exists a unique observable probability $P$ on $\cal E$ such that
\begin{equation}
E [X] = \int X(\omega) P (d \omega) .
\end{equation}
\end{theo}
We prove Theorem~\ref{TREP} in section~\ref{SREPR}.
In the final section we prove that the family of measurable functions
on a quasi-measurable space is an algebra of observables,
and give a concrete example on the non-linearity of
observable expectations.


\section{Integration.\label{SINT}}

The purpose of this section is to define the integral 
$\int X(\omega) P (d \omega)$ for certain observables $X$ when
$P$ is an observable probability.
The theory will be an extension of classical probability theory and
at the same time a generalization of the theory of quasi-measures 
\cite{AARNES:QUASI}.
Our proof of the quasi-linearity of the integral is 
much shorter than the original proof in the compact case
\cite[p.46-52]{AARNES:QUASI}.
The original proof uses the theory of the Riemann-Stieltjes integral,
but we get a simpler proof by avoiding this.

The family $\cal E$ of observable sets corresponds to the family of
events in classical probability theory.
The main difference is that $\cal E$ contains two types of events,
and the union of two observable sets of different type need not 
be an observable set.
The following result is as close as we can get to the $\sigma$-algebra axioms.

\begin{lemma}
The families ${\cal E}_U$ and ${\cal E}_F$ of co-zero and zero sets
are closed under finite unions and intersections.
\end{lemma}
\begin{proof}
Let $F$ and $G$ be zero sets with
$F = X^{-1}\{0\}$ and $G = Y^{-1}\{0\}$.
We may assume that $X,Y \ge 0$ since 
$\phi (t) = \abs{t}$ is continuous.
The observatons
$F \cap G = (X + Y)^{-1} \{0\}$ and
$F \cup G = (X \cdot Y)^{-1} \{0\}$,
and the continuity of 
$\phi (s,t) = s + t$ and 
$\phi (s,t) = s \cdot t$
give that ${\cal E}_F$ is closed under union and intersection.
The claim for ${\cal E}_U$ follows by taking the complement of
the above set equalities. 
\end{proof}

If the algebra of observables $\cal O$ is closed under uniform
convergence,
then ${\cal E}_F$ is closed under countable intersections and
${\cal E}_U$ is closed under countable unions.
This follows from consideration of the function
$X = \sum_{i = 1}^\infty 2^{-i} X_i$.

The family $\cal E$ of observable sets has its name from the
following result.
\begin{lemma}
\label{LOBS}
Let $X$ be an observable.
If the set $F \subset \RealN$ is closed,
then $X^{-1} (F)$ is a zero set.
If the set $U \subset \RealN$ is open,
then $X^{-1} (U)$ is a co-zero set.
\end{lemma}
\begin{proof}
The function $\phi (t) = d (t,F)$ is continuous and $F = \phi^{-1} \{0\}$.
This gives $X^{-1} (F) = \phi(X)^{-1}\{0\}$,
so  $X^{-1} (F)$ is a zero set.
From $X^{-1} (U) = [X^{-1} (U^c)]^c$ we conclude that
 $X^{-1} (U)$ is a co-zero set.
\end{proof}

The key to our definition of the integral is given by:
\begin{lemma}\label{LDIS}
Let $P$ be an observable probability and let $X$ be an observable.
There exists a unique Borel probability $P_X$ such that
$P_X (U) = P \circ X^{-1} (U)$ for all open sets $U \subset \RealN$.
\end{lemma}
\begin{proof}
The family of open sets is closed under finite intersections
and generates the Borel $\sigma$-field.
This gives uniqueness if we can prove existence.
The function $F_X (t) \dlik P \circ X^{-1} (-\infty,t]$
is well defined from Lemma \ref{LOBS}.
Definition \ref{SDEF}.\ref{DPROB} gives that $F_X$
is a right continuous non-decreasing function,
and furthermore that $F_X$ is the distribution function
of a Borel probability $P_X$, 
so $F_X (t) = P_X (-\infty,t]$.
An open set $U$ is a countable disjoint union of open intervals 
so the equality $P_X (U) = P \circ X^{-1} (U)$ follows from
\begin{equation}
P_X (\alpha,\beta) = 
\lim_{\beta_i \uparrow \beta} F_X (\beta_i) - F_X (\alpha) =
P\circ X^{-1} (-\infty,\beta) - P\circ X^{-1} (-\infty,\alpha] =
P\circ X^{-1} (\alpha,\beta) .
\end{equation}
\end{proof}

\begin{defin}
The probability measure $P_X$ in Lemma \ref{LDIS} is the distribution
of $X$.
The observable $X$ is integrable if $\id (t) = t$ is integrable
with respect to $P_X$, and then we define
\begin{equation}\int X (\omega) P (d\omega) \dlik \int t P_X (dt).\end{equation}
\end{defin}

Our next aim is to prove quasi-linearity of the integral.
The following definition is central and 
should be compared with equation (\ref{EINCL}) in section~\ref{SCOMP}.
\begin{defin}
A family $\{Q_X\}, X \in {\cal O},$ of Borel probabilities
on $\RealN$ is consistent if\\ 
$Q_{\phi (X)} (B) = Q_X \circ \phi^{-1} (B)$
for all Borel sets $B \subset \RealN$,
all $X \in {\cal O}$,
and all $\phi \in C (\im X)$.
\end{defin}

Our results imply that there is a one-one correspondence between
observable probabilities and consistent families of Borel 
probabilities.
We state the first half of this claim. 

\begin{lemma}
The family $\{P_X\}, X \in {\cal O},$ of distributions is consistent.
\end{lemma}
\begin{proof}
The family of open sets is closed under finite intersections and
generates the Borel $\sigma$-field so
$P_{\phi (X)} = P_X \circ \phi^{-1}$
follows from
$P_{\phi (X)} (U) = P \circ \phi(X)^{-1} (U) = 
 P \circ X^{-1} \circ \phi^{-1} (U) = 
P_X \circ \phi^{-1} (U)$.
\end{proof}
Quasi-linearity of the integral $\int X \, dP$ is an imediate consequence of the following Proposition.
\begin{pro}
Let $X \in {\cal O}$ and $\phi \in C (\im X)$.
If $\phi (X)$ is integrable with respect to $P$, then
\begin{equation}
\int \phi(X (\omega)) P(d\omega) = \int \phi (t) P_X (dt) .
\end{equation}
\end{pro}
\begin{proof}
$\int t P_{\phi (X)} (dt) = 
\int t P_{X} \circ \phi^{-1} (dt) = 
\int \phi (t) P_X (dt)$.
\end{proof}

The last inequality is a fundamental result in classical
probability theory.
The application of this result simplifies the proof
compared too the original one in the compact case \cite{AARNES:QUASI}.
The original proof takes the following Lemma as its starting point.
We need it in order to prove the monotone convergence property of
the integral.
\begin{lemma}
Assume that $X$ is integrable and that $X (\omega) \ge \alpha$
for all $\omega$.
We have
\begin{equation}
\int X (\omega) P (d\omega) =
\alpha + \int_\alpha^\infty P (X > t) dt =
\alpha + \int_\alpha^\infty P (X \ge t) dt .
\end{equation}
\end{lemma}
\begin{proof}
This follows from an application of the Fubini theorem
on the product measure from $P_X$ and the Lebesgue measure on $\RealN$.
\end{proof}
\begin{pro}
If $X_1 \ge \alpha$, $X$ is integrable, and $X_n \uparrow X$ pointwise, then
\begin{equation}
\int X_n (\omega) P (d\omega)\;\; \uparrow\;\;
\int X (\omega) P (d\omega) .
\end{equation}
\end{pro}
\begin{proof}
For each $t$ axiom {(IV)} in Definition~\ref{DPROB} of $P$
gives $P (X_n > t) \uparrow P (X > t)$.
The claim is then a consequence of the previous Lemma and
the Lebesgue monotone convergence theorem. 
\end{proof}

In particular we have now proven that integration with respect to
an observable probability gives an observable expectation.

\section{Representation of observable expectations.\label{SREPR}}

In this section we will prove that any observable expectation $E$
on an algebra $\cal O$ of observables is represented as the integral
$E [X] = \int X\; dP$ with respect to 
a unique observable probability $P$.
The path is given by proving:
(i) Monotonicity and uniform continuity of $E$.
(ii) An observable probability $P$ is determined by its integral
$\int X \; dP$.
(iii) An observable probability $P$ is determined by $E$.
(iv) The functional $E$ is equal to the functional obtained by
integration with respect to $P$.
 
The following Staircase Lemma is fundamental.
The statement is new, 
but the proof is very close to the proof of a corresponding
result by Aarnes \cite[p.54]{AARNES:QUASI}.

\begin{lemma}
Let $X \le Y$ in ${\cal O}^\infty$.
For each $\delta > 0$
we have the decomposition
$X = X_1 + \cdots + X_n$,
$Y = Y_1 + \cdots + Y_n$,
$X_i \in {\cal A}_X  \cap  {\cal A}_{X_i + Y_i}$,
$Y_i \in {\cal A}_Y  \cap  {\cal A}_{X_i + Y_i}$,
and 
$X_i \le Y_i + \delta / n$.
\end{lemma}
\begin{proof}
Choose a constant $M$ such that
$\tilde{X} \dlik X + M$, 
$\tilde{Y} \dlik Y + M + \delta$ obeys
$0 \le \tilde{X} \le \tilde{Y} - \delta$.
Choose $0 = \beta_0 < \cdots < \beta_n = \beta \dlik \norm{\tilde{Y}}$, 
$\beta_{i + 1} - \beta_i < \delta$, and define
\begin{equation}
\phi (x) \dlik \left\{
\begin{array}{ll}
0 & x \le 0\\
x & 0 \le x \le \beta\\
\beta & x \ge \beta
\end{array}
 \right. , \;\;\;\;\;\;
\phi_i (x) \dlik \left\{
\begin{array}{ll}
0 & x \le \beta_{i - 1}\\
x -  \beta_{i - 1}  &  \beta_{i - 1}  \le x \le \beta_i \\
\beta_i - \beta_{i - 1} & x \ge \beta_i
\end{array}
 \right. .
\end{equation}
With $X_i \dlik \phi_i (\tilde{X}) - M/n$, 
$Y_i \dlik \phi_i (\tilde{Y}) - (M + \delta)/n$,
and the observation $\phi = \sum_i \phi_i$, 
we conclude $X = \sum_i X_i$, $Y = \sum_i Y_i$,
and $X_i \in {\cal A}_{X}, Y_i \in {\cal A}_{Y}$.
We prove $X_i, Y_i \in {\cal A}_{X_i + Y_i}$,
or equivalently  
$\tilde{X}_i \dlik  \phi_i (\tilde{X}), 
\tilde{Y}_i \dlik  \phi_i (\tilde{Y}) \in {\cal A}_{\tilde{X}_i + \tilde{Y}_i}$.
From $\tilde{X} \le \tilde{Y} - \delta$ 
we conclude $\tilde{Y}_i (x) = \beta_i - \beta_{i - 1}$
when $\tilde{X}_i (x) > 0$.
This gives $\tilde{X}_i \cdot (\beta_i - \beta_{i - 1} - \tilde{Y}_i) = 0$,
$\tilde{X}_i, \beta_i - \beta_{i - 1} - \tilde{Y}_i 
\in {\cal A}_{\tilde{X}_i - \beta_i + \beta_{i - 1} + \tilde{Y}_i}$, 
and finally
$\tilde{X}_i, \tilde{Y}_i \in {\cal A}_{\tilde{X}_i + \tilde{Y}_i}$.
\end{proof}

\begin{pro}
If $E : {\cal O}^\infty \into \RealN$ is positive
and quasi-linear, then
$X \le Y \imply E [X]    \le E [Y]   $, and
$\abs{E [X]    - E [Y]   } \le E [1]      \norm{X - Y}$.
\end{pro}
\begin{proof}
The Staircase Lemma gives
$E[X] = \sum_i E[X_i] 
\le \delta E[1]   + \sum_i E[Y_i] = \delta E[1]   + E[Y]$
from which we conclude $E[X] \le E[Y]$.
The observation $X \le Y + \norm{X - Y}$ 
gives $E[X] \le E[Y] + E[1  ] \norm{X - Y}$
and a switch of $Y$ and $X$ gives
$\abs{E[X] - E[Y]} \le E[1  ] \norm{X - Y}$.
\end{proof}

It follows in particular that an observable expectation is monotonic and continuous.
The next aim is to prove that an observable probability is
determined by its integral.
We need some preliminaries.

Let $A$ be a subset of $\Omega$ and $X$ be a function on $\Omega$.
If $X \le 1_A$, we write $X \le A$.
The inequality
$A \le X$ is interpreted in a similar way by replacing the set $A$ with its
indicator function $1_A$.
The notation $X \preceq A$ means that there exists a zero set $F$
such that $X \le F \subset A$.
The proof of the following Urysohn Lemma is the only place where we
really need that $\cal O$ is something less general than an algebra of
functions.

\begin{lemma}\label{LURY}
If $F$ is a zero set contained in a co-zero set $U$, 
there exists an observable $X$ such that $F \le X \preceq U$.
\end{lemma}
\begin{proof}
Let $F = Y^{-1} \{0\}$ and $U^c = Z^{-1} \{0\}$ with $Y,Z \ge 0$.
The equality $\phi (y,z) = z / (y + z)$ defines a function
$\phi \in C (\im (Y,Z))$,
so $\tilde{X} \dlik \phi(Y,Z) \in {\cal O}$.
It follows that 
$\tilde{X}^{-1} \{1\} = F$,
$\tilde{X}^{-1} \{0\} = U^c$, and in particular
$F \le \tilde{X} \le U$.
With $V \dlik \tilde{X}^{-1} (1/2,1]$ and
$G \dlik \tilde{X}^{-1} [1/2,1]$ we conclude
$F \subset V \subset G \subset  U$.
A repetition of the above argument gives an observable $X$
with $F \le X \le V$ which gives the claim.
\end{proof}

A particular consequence is that disjoint zero sets are separated by
(disjoint) co-zero sets.
The monotonicity property of $P$ is extended by:
\begin{lemma}\label{LMON}
We have $P (F) \le \int X \;dP \le P (U)$ if $F \le X \le U$.
\end{lemma} 
\begin{proof}
$P (U) \ge P (X > 0) = \int_{(0,1]} P_X (dt) \ge
\int_{(0,1]} t P_X (dt) \ge 
\int_{\{1\}} P_X (dt) \ge P (F)$.
\end{proof}

We write $U_i \uparrow_r U$ if $U_i \preceq U_{i + 1}$ and
$U = \cup_i U_i$ in ${\cal E}_U$, and in this case we say
that $U$ is a regularized limit of the sequence $(U_i)$.
\begin{lemma}
Every co-zero set $U$ is a regularized limit of a sequence $(U_i)$ of
co-zero sets. 
\end{lemma}
\begin{proof}
We may assume that $U = X^{-1} (0,\infty)$, $X \ge 0$.
By defining $U_i \dlik X^{-1} (1/i,\infty)$ and
$F_i \dlik X^{-1} [1/i,\infty)$,
we get $U_i \subset F_i \subset U_{i + 1}$ and $U = \cup_i U_i$.
\end{proof}

We are now in a position to prove that any observable probability
is a quasi-additive probability, 
and that the integral determines
the probability:

\begin{pro}\label{PPRO}
If $U$ is a co-zero set, then
\begin{equation}
P (U) = \sup \{P (F) \st F \subset U,\; F \in {\cal E}_F\} =
\sup_{X \preceq U} \int X\;dP
.\end{equation} 
\end{pro}
\begin{proof}
Let $U_i \uparrow_r U$ with $U_i \subset F_i \subset U_{i + 1}$.
By definition $P (U_i) \uparrow P (U)$ and the monotonicity
gives $P (F_i) \uparrow P (U)$ and the first equality in the claim.
The second equality in the claim follows from this and
Lemmas \ref{LMON} and \ref{LURY}. 
\end{proof}

We will now prove that a quasi-additive expectation $E$
determines a quasi-additive probability.
This result is more general than we need,
but we include the statement since its a first step in a
proof of a representation theorem for quasi-additive expectations
by integration.

\begin{pro}\label{PQUA}
Let $E$ be a quasi-additive expectation on an algebra $\cal O$
of observables.
The definition $P (U) \dlik \sup \{E [X] \st X \preceq U\}$ determines
a unique  quasi-additive probability $P$.
\end{pro}
\begin{proof}
In order to prove additivity we will prove
the preliminary result 
$\sup \{E [X] \st X \le U\} = \sup \{E [Y] \st Y \preceq U\}$.
Let $X \le U$ and $\epsilon > 0$.
Lemma \ref{LURY} gives $(X \ge \epsilon) \le Z \preceq U$.
Consequently $Y \dlik X \cdot Z \preceq U$,
and $E [X] - E [Y] \le \norm{X - Y} < \epsilon$ implies 
$\sup \{E [X] \st X \le U\} = \sup \{E [Y] \st Y \preceq U\}$.\\
{\bf Axiom  (I)}. 
If $X \le U$ and $Y \le V$ with $U \cap V = \emptyset$,
then $X,Y \in {\cal A} (X - Y)$,
and $E [X + Y] = E [X] + E [Y]$.
From this we have $X + Y \le U \uplus V$ and we have the inequality
$P (U \uplus V) \ge P (U) + P (V)$.
In order to prove equality we assume $X \le F \subset U \uplus V$.
This gives $F = (F \cap V^c) \uplus (F \cap U^c)$,
and together with Lemma \ref{LURY} this gives
$(F \cap V^c) \le Y \preceq U$ and 
$(F \cap U^c) \le Z \preceq V$.
Finally $X = X \cdot Y + X \cdot Z$ ensures
$P (U \uplus V) \le P (U) + P (V)$,
and additivity on ${\cal E}_U$ is proven.
We can then consistently extend $P$ to ${\cal E}_F$
by $P (F) \dlik 1 - P (F^c)$.
It follows that $P (F) = \inf \{E [X] \st F \le X\}$ since
$F \le X$ is equivalent with $F^c \ge 1 - X$.
From this and Lemma \ref{LURY} we conclude that $P$ is additive on 
${\cal E}_F$.
Additivity on ${\cal E}_U$ and ${\cal E}_F$ together with
$1 = P (F) + P (F^c)$ give additivity on ${\cal E}$.\\
{\bf Axiom  (IV)'}.
We must prove $P (U) = \sup \{P (F) \st F \subset U,\; F \in {\cal E}_F\}$.
If $F \subset U$, then $U = F \uplus (U \cap F^c)$,
and additivity gives $P (F) \le P (U)$.
It is therefore sufficient to prove that $X \le G \subset U$ gives
$E [X] \le  \sup \{P (F) \st F \subset U,\; F \in {\cal E}_F\}$.
This is however a consequence of 
$E [X] \le \inf \{E [Y] \st G \le Y\} = P (G)$ where
the inequality follows from the monotonicity of $E$.\\
{\bf Axioms  (II-III)} are evidently satisfied.
\end{proof}

\begin{proof} {\bf (Theorem \ref{TREP})}
We define $P (U) = \sup \{E [X] \st X \preceq U \}$ and Proposition \ref{PQUA}
tells us that this determines a quasi-additive probability $P$.
Step 1 is to prove that $P$ is an observable probability,
and step 2 is to prove  $E [X] = \int X\; dP$.
The uniqueness of $P$ from this is a consequence
of Proposition \ref{PPRO}.\\
{\bf Step 1}
We must prove that $U_i \uparrow U$ implies $P (U_i) \uparrow P (U)$.
Let $V_{i,n} \uparrow U_i$ with $V_{i,n} \subset F_{i,n} \subset V_{i,n + 1}$.
The choices $F_{i,n} \le \tilde{X}_{i,n} \le U_i$ and
$X_{i,n} \dlik \tilde{X}_{i,1} \vee \cdots \vee  \tilde{X}_{i,n}$
gives $X_{i,n} \uparrow U_i$.
Let $X \le U$ and define $X_i = X \cdot (X_{1,i} \vee \cdots \vee X_{i,i})$.
It follows that $X_i \uparrow X$ and $X_i \le U_i$.
This gives $E [X_i] \le P (U_i)$ and $E [X] \le \lim P (U_i)$.
The equality $P (U) = \lim P (U_i)$ follows finally
from $P (U) = \sup \{E [X] \st X \le U\} \le \lim P (U_i) \le P (U)$.\\
{\bf Step 2}
The Riesz representation theorem \cite{RU:ANAL}
and the quasi-linearity of $E$ gives a
Borel probability $Q_X$ such that 
$E [\phi (X)] = \int \phi (t) \; Q_X (dt)$.
We need only prove that $Q_X = P_X$.
Let $U$ be an open subset of $\RealN$.
From the argument in step 1 it follows that we can
find $\phi_i \in C (\RealN)$ with $\phi_i \uparrow U$.
It follows that $\phi_i (X) \uparrow X^{-1} (U)$ and
the monotone convergence theorem \cite{RU:ANAL} gives
$Q_X (U) = \lim E [\phi_i (X)] 
\stackrel{(*)}{=} \sup \{E [Y] \st Y \le X^{-1} (U)\} = P_X (U)$
which proves the claim.
The equality $\stackrel{(*)}{=}$ follows from a more general statement and
will be proven in the following Lemma. 
\end{proof}

\begin{lemma}
Let $E$ be an observable expectation with a corresponding
observable probability $P$.
We have 
\begin{equation}
P (U) = \lim_{X_i \uparrow U} E [X_i].
\end{equation}
\end{lemma}
\begin{proof}
The assumption is that the sequence $X_i$ converges pointwise
and monotonic up to the co-zero set $U$.
We will prove $\sup \{E [X] \st X \le U\} = \lim E [X_i]$.
The inequality $\ge$ follows from  $\sup \{E [X] \st X \le U\} \ge E [X_i]$.
The inequality $\le$ follows if we can prove that $X \le U$
gives $E [X] \le \lim E [X_i]$,
but this follows from $E [X] = \lim E [X \cdot X_i]$ and the
monotonicity of $E$.
\end{proof}

\section{Quasi-measurable spaces.}

In this section we present the generalization of the
$\sigma$-algebra approach to integration, 
and give an example which shows that the quasi-integral need not be linear.

Let ${\cal F}_U$ be a family of subsets in a set $\Omega$,
and let ${\cal F}_F$ be the family of complements,
${\cal F}_U$, ${\cal F}_F = \{F \st F^c \in {\cal F}_U \}$.
The pair $({\cal F}_U, {\cal F}_F)$ is said to be
a $\sigma$-quasi-algebra in $\Omega$
if ${\cal F}_U$ is closed under intersections and countable unions.
By abuse of language we will say that the 
union ${\cal F} = {\cal F}_U \cup {\cal F}_F$ of the two families
is a  $\sigma$-quasi-algebra in $\Omega$.
A quasi-measurable space is a set $\Omega$ equipped
with a $\sigma$-quasi-algebra.
A function $X: \Omega \into \RealN$ on a 
quasi-measurable space $\Omega$
is measurable if
$X^{-1} (U) \in {\cal F}_U$ for all open sets
$U \subset \RealN$.

\begin{lemma}
The family $\cal O$ of measurable functions 
$X: \Omega \into \RealN$ on a 
quasi-measurable space $\Omega$ is an algebra of observables.
\end{lemma}
\begin{proof}
Let $U$ be an open subset of $\RealN$, 
$X_1, \ldots , X_n \in {\cal O}$, and 
$\phi \in C (\im (X_1,\ldots , X_n ))$.
The continuity of $\phi$ gives an open set $V \subset \RealN^n$
such that $\phi^{-1} (U) = V \cap \im (X_1,\ldots , X_n )$.
The real line is second countable, so we have
a representation $V = \cup_{i = 1}^\infty (U_i^1 \times \cdots \times U_i^n)$,
and we get
$\phi (X_1, \ldots , X_n)^{-1} (U) = 
\cup_{i = 1}^\infty [X_1^{-1}(U_i^1) \cap \cdots \cap X_1^{-1} (U_i^n)]
\in {\cal F}_U$.
This proves $\phi (X_1, \ldots , X_n)\in {\cal O}$,
and $\cal O$ is an algebra of observables.
\end{proof}

There exist examples that show that the 
family $\cal E$ of observable sets may
be a proper subfamily of the family $\cal F$.
This is our main reason for taking an algebra of observables as
the starting point.

Let $P$ be the Aarnes measure \cite{AARNES:QUASI}[p.59-64] 
on the square $\Omega = [-1,1]^2$
defined by the point $\omega_P = (0,0)$ and the curve $\gamma_P$,
where $\gamma_P$ is the border of $\Omega$.
It follows that
\begin{equation}
\int (\alpha \omega_1^2 + \beta \omega_2^2) P (d\omega) = \alpha \wedge \beta,
\;\; \alpha, \beta \ge 0, 
\end{equation} 
from consideration of the level curves of the integrand. 
From this it is clear that the integral is highly non-linear,
and the observables 
$X_1 (\omega) = \omega_1$ and
$X_2  (\omega) = \omega_2$ are incompatible.

It should be noted that the presented 
theory of nonlinear measures initiates
from a generalized $C^*$-algebraic 
formulation of quantum mechanics \cite{AARNES:PHYS}.   
This gives a program for possible extension and application of the theory.
In this article we have tried to clearify parts of the 
theory corresponding to a commutative $C^*$-algebra.

\vspace{2ex}
{\rm ACKNOWLEDGEMENT.\ } 
I would like to give my thanks to J.F. Aarnes 
for discussions while preparing this paper.

{\small
\bibliography{bib}

\begin{thebibliography}{1}

\bibitem{AARNES:PHYS}
J.F. Aarnes.
\newblock Physical states on a {$C^*$}-algebra.
\newblock {\em Acta Math.}, 122:161--172, 1969.

\bibitem{AARNES:QUASI}
J.F. Aarnes.
\newblock Quasi-states and quasi-measures.
\newblock {\em Adv. in Math.}, 86(1):41--67, 1991.

\bibitem{AarnesRustad99qprob}
J.F. Aarnes and A.B. Rustad.
\newblock Probability and quasi-measures - a new interpretation.
\newblock {\em MATH. SCAND.}, 85, 1999.

\bibitem{JAUCH}
J.M. Jauch.
\newblock {\em Foundations of quantum mechanics}.
\newblock Addison-Wesley, 1968.

\bibitem{KOLMOGOROV}
A.~Kolmogorov.
\newblock Grundbegriffe der {W}ahrscheinlichkeitsrechnung.
\newblock {\em Erg. Mat.}, 2(3), 1933.

\bibitem{NEU:QM}
J.~{von} Neumann.
\newblock {\em Mathematical foundations of quantum mechanics.}
\newblock Princeton, 1967 (1932).

\bibitem{RU:ANAL}
W.~Rudin.
\newblock {\em Real and Complex Analysis}.
\newblock McGraw-Hill, USA, third edition, 1987.

\bibitem{TARALDSEN14}
G.~Taraldsen.
\newblock Om kvantemekanikk og kunsten {\aa} gjette.
\newblock {\em Symmetri}, 2(1-2):23--29, 1995.

\end{thebibliography}
\bibliographystyle{plain}
}

\end{document}